\documentclass[12pt]{amsart}

\usepackage{epsfig,color}
\usepackage{amssymb}
\newtheorem{theorem}{Theorem}[section]

\newtheorem{lemma}[theorem]{Lemma}

\newtheorem{corollary}[theorem]{Corollary}

\title{{\bf Cyclic $m$-Cycle Systems of Complete Graphs minus a 1-factor}}

\author{Heather Jordon}
\address{Department of Mathematics\\
Illinois State University\\
Campus Box 4520\\
Normal, IL 61790-4520\\
USA} \email{hjordon@ilstu.edu}

\author{Joy Morris}
\address{Department of Mathematics \& Computer Science\\
University of Lethbridge\\
 Lethbridge, AB \\
Canada T1K 3M4}
\email{joy@cs.uleth.ca}

\date{\today}

\begin{document}
\def \l {\lambda}
\def \mod {{\rm mod}\,}
\def \Z {\mathbb Z}

\begin{abstract}
In this paper, we provide necessary and sufficient conditions for
the existence of a cyclic $m$-cycle system of $K_n-I$ when $m$ and
$n$ are even and $m \mid n$.
\end{abstract}

\maketitle
\def\baselinestretch{1.0}\small\normalsize

\section{Introduction}

Throughout this paper, $K_n$ will denote the complete graph on $n$
vertices, $K_n - I$ will denote the complete graph on $n$ vertices
with a 1-factor $I$ removed (a 1-{\it factor} is a 1-regular
spanning subgraph), and $C_m$ will denote the $m$-cycle $(v_1, v_2,
\ldots, v_m)$.  An {\it $m$-cycle system} of a graph $G$ is a set
$\mathcal C$ of $m$-cycles in $G$ whose edges partition the edge set
of $G$. An $m$-cycle system is called {\it hamiltonian} if $m =
|V(G)|$.

Several obvious necessary conditions for an $m$-cycle system
$\mathcal C$ of a graph $G$ to exist are immediate: $m \le |V(G)|$,
the degrees of the vertices of $G$ must be even, and $m$ must divide
the number of edges in $G$. A survey on cycle systems is given in
\cite{Ro} and necessary and sufficient conditions for the existence
of an $m$-cycle system of $K_n$ and $K_n-I$ were given in
\cite{AlGa,Sa} where it was shown that a $m$-cycle system of $K_n$
or $K_n-I$ exists if and only if $n\geq m$, every vertex of $K_n$ or
$K_n - I$ has even degree, and $m$ divides the number of edges in
$K_n$ or $K_n - I,$ respectively.

Throughout this paper, $\rho$ will denote the permutation $(0\ 1\
\ldots\ n-1)$, so $\langle\rho\rangle=\mathbb Z_n$.  An $m$-cycle
system $\mathcal C$ of a graph $G$ with vertex set $V(G)=\mathbb
Z_n$ is $cyclic$ if, for every $m$-cycle $C=(v_1, v_2, \ldots, v_m)$
in $\mathcal C$, the $m$-cycle $\rho(C)=(\rho(v_1), \rho(v_2),
\ldots, \rho(v_m))$ is also in $\mathcal C$. An $n$-cycle system
$\mathcal C$ of a graph $G$ with vertex set $\mathbb Z_n$ is called
a {\it cyclic hamiltonian cycle system.} Finding necessary and
sufficient conditions for cyclic $m$-cycle systems of $K_n$ is an
interesting problem and has attracted much attention (see, for
example, \cite{BEV, BG1, BD1, BD2, EPV, FW, Ko, Rosa}). The obvious
necessary conditions for a cyclic $m$-cycle system of $K_n$ are the
same as for an $m$-cycle system of $K_n$; that is, $n \ge m \ge 3$,
$n$ is odd (so that the degree of every vertex is even), and $m$
must divide the number of edges in $K_n$. However, these conditions
are no longer necessarily sufficient. For example, it is not
difficult to see that there is no cyclic decomposition of $K_{15}$
into 15-cycles. Also, if $p$ is an odd prime and $\alpha \ge 2$,
then $K_{p^{\alpha}}$ cannot be decomposed cyclically into
$p^{\alpha}$-cycles \cite{BD2}.

The existence question for cyclic $m$-cycle systems of $K_n$ has
been completely settled in a few small cases, namely $m = 3$
\cite{Pe}, $5$ and $7$ \cite{Rosa}. For even $m$ and $n \equiv
1 \pmod {2m}$,
cyclic $m$-cycle systems of $K_{n}$ are constructed for $m
\equiv 0 \pmod 4$ in \cite{Ko} and for $m \equiv 2 \pmod 4$ in
\cite{Rosa}. Both of these cases are handled simultaneously in
\cite{EPV}. For odd $m$ and $n \equiv 1 \pmod {2m}$, cyclic $m$-cycle
systems of $K_{n}$ are found using different methods in \cite{BEV,
BD1, FW}. In \cite{BG1}, as a consequence of a more general result,
cyclic $m$-cycle systems of $K_{n}$ for all positive integers $m$
and $n \equiv 1\ \pmod {2m}$ with $ n\ge m \ge 3$ are given using
similar methods. In \cite{BD2}, it is shown that a cyclic
hamiltonian cycle system of $K_n$ exists if and only if $n \not=15$
and $n \not\in \{p^{\alpha} \mid p \mbox{ is an odd prime and }
\alpha \ge 2 \}$. Thus, as a consequence of a result in \cite{BD1},
cyclic $m$-cycle systems of $K_{2mk+m}$ exist for all $m \not=15$
and $m \not\in \{p^{\alpha} \mid p \mbox{ is an odd prime and }
\alpha \ge 2 \}$. In \cite{V1}, the last remaining cases for cyclic
$m$-cycle systems of $K_{2mk+m}$ are settled, i.e., it is shown
that, for $k \ge 1$, cyclic $m$-cycle systems of $K_{2km+m}$ exist
if $m = 15$ or $m \in \{p^{\alpha} \mid p \mbox{ is an odd prime and
} \alpha \ge 2 \}$. In  \cite{WF}, necessary and sufficient
conditions for the existence of cyclic $2q$-cycle  and $m$-cycle
systems of the complete graph are given when $q$ is an odd prime
power and $3 \le m \le 32$. In \cite{B}, cycle systems with a
sharply vertex-transitive automorphism group that is not necessarily
cyclic are investigated. As a result, it is shown in \cite{B} that
no cyclic $m$-cycle system of $K_n$ exists if $m < n < 2m$ with $n$
odd and gcd$(m,n)$ a prime power.  In \cite{Wu}, it is shown that
if $m$ is even and $n > 2m$, then there exists a cyclic $m$-cycle system
of $K_n$ if and only if the obvious necessary conditions that $n$ is odd and that $n(n-1) \equiv 0 \pmod {2m}$ hold.

These questions can be extended to the case when $n$ is even by
considering the graph $K_n - I$. In \cite{BG1}, it is shown that for
all integers $m\geq 3$ and $k\geq 1$, there exists a cyclic
$m$-cycle system of $K_{2mk+2}-I$ if and only if $mk\equiv 0,3 \pmod
4$.  In \cite{GM}, it is shown that for an even integer $n \ge 4$,
there exists a cyclic hamiltonian cycle system of $K_n - I$ if and
only if $n \equiv 2, 4 \pmod 8$ and $n \not= 2p^{\alpha}$ where $p$
is an odd prime and $\alpha \ge 1$.  In \cite{BR}, it was shown that in every cyclic cycle decomposition of $K_{2n} - I$, the number of cycle orbits of odd length must have the same partity as $n(n-1)/2$.  As a consequence of this result, in \cite{BR}, it is shown that a cyclic $m$-cycle system of $K_{2n}-I$ can not exist if $n \equiv 2,3 \pmod 4$ and $m \not\equiv 0 \pmod 4$ or $n \equiv 0,1 \pmod 4$ and $m$ does not divide $n(n-1)$. In this paper we are interested in
cyclic $m$-cycle systems of $K_n-I$ when $m$ and $n$ are even and
$m \mid n$.  The main result of this paper is the following.

\begin{theorem} \label{main} For an even integer $m$ and integer $t$, there
exists
a cyclic $m$-cycle system of $K_{mt}-I$ if and only if
\begin{enumerate}
\item $t \equiv 0,2 \ (\mod 4)$ when $m \equiv 0\ (\mod 8),$
\item $t \equiv 0,1 \ (\mod 4)$ when $m \equiv 2\ (\mod 8)$ with
$t > 1$ if $m = 2p^\alpha$ for some prime $p$ and integer $\alpha
\ge 1,$
\item $t \ge 1$ when $m \equiv 4\ (\mod 8)$, and
\item $t \equiv 0,3 \ (\mod 4)$ when $m \equiv 6\ (\mod 8)$.
\end{enumerate}
\end{theorem}

Our methods involve circulant graphs and difference constructions.
In Section 2, we give some basic definitions and lemmas while the
proof of Theorem \ref{main} is given in Sections 3, 4 and 5.  In
Section 3, we handle the case when $m \equiv 0\pmod 8$ and show
that there is a cyclic $m$-cycle system of $K_{mt}-I$ if and only if
$t \ge 2$ is even.  In Section 4, we handle the case when $m \equiv 4 \pmod
8$ and show that there is a cyclic $m$-cycle system of $K_{mt}-I$
if and only if $t \ge 1$. In Section 5, we handle the case when $m \equiv 2 \pmod 4$.  When $m \equiv 2 \pmod 8$, we
show that there is a cyclic $m$-cycle system of $K_{mt}-I$
if and only if $t \equiv 0, 1 \pmod 4$.  When $m \equiv 6 \pmod 8$, we show that there is a cyclic $m$-cycle system of $K_{mt}-I$
if and only if $t \equiv 0, 3\pmod 4$.  Our main theorem then
follows.

\section{Preliminaries}

The notation $[1, n]$ denotes the set $\{1, 2, \ldots, n\}$. The proof of Theorem \ref{main} uses circulant graphs, which we now
define.  For $x\not\equiv 0\ (\mod n)$, the {\it modulo $n$ length} of an
integer $x$, denoted $|x|_n$, is
defined to be the smallest positive integer $y$ such
that $x\equiv y\ (\mod n)$ or $x\equiv -y\ (\mod n)$. Note that for any
integer $x\not\equiv 0\ (\mod n)$, it follows that
$|x|_n\in [1, \lfloor \frac{n}{2}\rfloor]$.
If $L$ is a set of
modulo $n$ lengths, we define the {\it circulant graph} $\langle L\rangle_n$ to be the graph with
vertex set $\mathbb Z_n$ and edge set
$\{\{i,j\}\mid|i-j|_n\in L\}$.  Notice that in order for a graph $G$ to admit
a cyclic $m$-cycle decomposition, $G$ must be a circulant graph, so
circulant graphs provide a natural setting in which to construct
cyclic $m$-cycle decompositions.

The graph $K_n$ is a circulant graph, since $K_n = \langle \{1, 2,
\ldots, \lfloor n/2 \rfloor\}\rangle_n$.  For $n$ even, $K_n - I$ is also a circulant graph,
since $K_n - I = \langle \{1, 2, \ldots, n-1\}\setminus\{n/2\}\rangle_n$ (so
the edges of the 1-factor $I$ are of the form $\{i, i+n/2\}$ for $i
= 0,1, \ldots, (n-2)/2$).

Let $H$ be a subgraph of a circulant graph $\langle L \rangle_n$. The notation $\ell(H)$ will denote the set
of modulo $n$ edge lengths belonging to $H$, that is,
$$\ell(H) = \{\ell \in L \mid \{g, g+\ell\} \in E(H) \text{ for some
}g\in \Z_n\}.$$ Many properties of $\ell(H)$ are independent of the
choice of $L$; in particular, the next lemma in this
section does not depend on the choice of $L$.

Let $C$ be an $m$-cycle in circulant graph $\langle L \rangle_n$ and recall that the permutation
$\rho = (0\ 1\
\ldots\ n-1)$, which generates $\Z_n$, has the property that $\rho(C)\in
\mathcal C$ whenever $C\in \mathcal C$.  We can therefore consider
the action of $\Z_n$ as a permutation group acting on the elements
of $\mathcal C$.  Viewing matters this way, the {\it length of the
orbit of} $C$ (under the action of $\Z_n$) can be defined as the
least positive integer $k$ such that $\rho^k(C) = C$. Observe that
such a $k$ exists since $\rho$ has finite order; furthermore, the
well-known orbit-stabilizer theorem (see, for example \cite[Theorem
1.4A(iii)]{DM}) tells us that $k$ divides $n$. Thus, if $G$ is a
graph with a cyclic $m$-cycle system $\mathcal C$ with $C \in
{\mathcal C}$ in an orbit of length $k$, then it must be that $k$
divides $n=|V(G)|$ and that $\rho(C), \rho^2(C), \ldots,
\rho^{k-1}(C)$ are distinct $m$-cycles in ${\mathcal C}$.

 The
next lemma gives many useful properties of an $m$-cycle $C$ in a
cyclic $m$-cycle system $\mathcal C$ of a graph $G$ with $V(G)=
\mathbb Z_n$ where $C$ is in an orbit of length $k$.  Many of these
properties are also given in \cite{BD2} in the case that $m=n$. The
proofs of the following statements follow directly from the previous
definitions and are therefore omitted.

\begin{lemma} \label{distinct_length}
Let $\mathcal C$ be a cyclic $m$-cycle system of a graph $G$ of
order $n$ and let $C \in {\mathcal C}$ be in an orbit of length $k$.
Then
\begin{enumerate}
\item $|\ell(C)| = mk/n$;
\item $C$ has $n/k$ edges of length $\ell$ for each $\ell \in \ell(C)$;
\item $(n/k) \mid \gcd(m,n)$;
\end{enumerate}
Let $k > 1$ and let $P: v_0=0, v_1, \ldots v_{mk/n}$ be a subpath of $C$ of
length $mk/n$. Then
\begin{enumerate}
\item [(4)] if there exists $\ell \in \ell(C)$
with $k \mid \ell$, then $m = n / \gcd(\ell,n)$,
\item [(5)] $v_{mk/n} = kx$ for some integer $x$ with $\gcd(x, n/k)=1$,
\item [(6)] $v_1, v_2, \ldots, v_{mk/n}$ are distinct modulo $k$,
\item [(7)] $\ell(P)=\ell(C)$, and
\item [(8)] $P, \rho^k(P), \rho^{2k}(P), \ldots, \rho^{n-k}(P)$ are
pairwise edge-disjoint subpaths of $C$.
\end{enumerate}
\end{lemma}

Let $X$ be a set of $m$-cycles in a graph $G$ with vertex set
$\mathbb Z_n$ such that ${\mathcal C}=\{\rho^i(C) \mid C\in
X,i=0,1,\dots,n-1\}$ is an $m$-cycle system of $G$. Then $X$ is
called a {\it generating set} for $\mathcal C$. Clearly, every cyclic $m$-cycle
system $\mathcal C$ of a graph $G$ has a generating set $X$ as we may
always let $X = {\mathcal C}.$  A generating set $X$ is called a {\it
minimum generating set} if $C \in X$ implies $\rho^i(C) \not\in X$ for
$1 \le i \le n$ unless $\rho^i(C) = C$.

Let $\mathcal C$ be a cyclic $m$-cycle system of a graph $G$ with
$V(G) = \mathbb Z_n$. To find a minimum generating set $X$ for
$\mathcal C$, we start by adding $C_1$ to $X$ if the length of the
orbit of $C_1$ is maximum among the cycles in $\mathcal C$.  Next,
we add $C_2$ to $X$ if the length of the orbit of $C_2$ is maximum
among the cycles in ${\mathcal C} \setminus \{ \rho^i(C_1) \mid 0
\le i \le n-1 \}$.  Continuing in this manner, we add $C_3$ to $X$
if the length of the orbit of $C_3$ is maximum among the cycles in
${\mathcal C} \setminus \{ \rho^i(C_1), \rho^i(C_2) \mid 0 \le i \le
n-1 \}$.  We continue in this manner until $\{\rho^i(C) \mid C \in
X, 0 \le i \le n-1 \} = {\mathcal C}$. Therefore, every cyclic
$m$-cycle system has a minimum starter set. Observe that if $X$ is a
minimum generating set for a cyclic $m$-cycle system $\mathcal C$
of the graph $\langle L \rangle_n$, then
it must be that the collection of
sets $\{ \ell(C) \mid C \in X\}$ forms a partition of $L$.

In this paper, we are interested in the cyclic $m$-cycle systems of
$K_n-I$ where $n = mt$ for some positive integer $t$.  Suppose $K_n$
has a cyclic $m$-cycle system $\mathcal C$ for some $n = mt$. Let
$X$ be a minimum generating set for $\mathcal C$ and let $C \in X$ be a
cycle in an orbit of length $k$.  Then, $\ell(C)$ has $mk/n = k/t$ lengths which
implies that $k = \ell t$ for some integer $\ell$. Also, since
$|\ell(C)| = \ell$, it follows that $\ell \mid m$. The following
lemma will be useful in determining the congruence classes of $t$
based on the congruence class of $m$ modulo 8.

\begin{lemma} \label{t-conditions}
Let $m$ be an even integer and let $K_{mt}-I$ have a cyclic
$m$-cycle system for some positive integer $t$.
\begin{enumerate}
\item If $\{1, 2, \ldots, (mt-2)/2\}$ has an odd number of even integers,
then $t$ is even.

\item If $\{1, 2, \ldots, (mt-2)/2\}$ has an odd number of odd integers,
then $t$ is odd.
\end{enumerate}
\end{lemma}

\begin{proof}
Let $m$ be even and suppose $K_{mt}-I$ has a cyclic $m$-cycle system
$\mathcal C$ for some positive integer $t$. Let $V(K_{mt}) = \mathbb
Z_{mt}$, and let $X$ be a minimum generating set for $\mathcal C$.

Suppose first that $\{1, 2, \ldots, (mt-2)/2\}$ has an odd number of
even integers. Since the set $\{\ell(C) \mid C \in X\}$ is a
partition of $\{1, 2, \ldots, (mt-2)/2\}$, there must be an odd
number of cycles $C$ in $X$ with $\ell(C)$ containing an odd number
of evens.  Let $C\in X$ be a cycle in an orbit of length $k$ with an odd number of even edge lengths.  Let
$|\ell(C)|=\ell$ and note that $k =\ell t$.  From Lemma \ref{distinct_length},
we know that the subpath of $C$ starting at vertex $0$ of length
$\ell$ ends at vertex $jk$ with $\gcd(j, m/\ell) = 1$.

Suppose first $k$ is odd. Then $\ell$ and $t$ must both be odd. Thus
$m/ \ell$ is even so that $jk$ is odd.  Hence, $\ell(C)$ contains an
odd number of odd integers and, since $|\ell(C)|$ is odd, an even
number of even integers,  contradicting the choice of $C$.
Thus, $k$ is even.  Since $k$ is even, $jk$ is even.  Thus,
$\ell(C)$ contains an even number of odd integers. If $\ell$ is
even, then $\ell(C)$ also contains an even number of even integers,
contradicting the choice of $C$.  Thus, $\ell$ is odd. Since $k$ is
even and $k = \ell t$, it must be that $t$ is even.

Now suppose $\{1, 2, \ldots, (mt-2)/2\}$ has an odd number of odd
integers.  Hence there are an odd number of cycles $C$ in $X$ with
$\ell(C)$ containing an odd number of odd integers.  Again, let $C\in X$
be such a cycle with $|\ell(C)| = \ell$, in an orbit of length $k= \ell t$. Let
the subpath of $C$ starting at vertex $0$ of length $\ell$ end at
vertex $jk$ with $\gcd(j, m/ \ell) = 1$.  Now, if $k$ is even, then
$jk$ is even so that $\ell(C)$ contains an even number of odd
integers, contradicting the choice of $C$.  Thus $k$ is odd.  Since $k=\ell t$, we have that $t$ is odd.
\end{proof}

The following corollary is an immediate consequence of Lemma
\ref{t-conditions} and \cite{GM}.

\begin{corollary}
For an even integer $m$ and a positive integer $t$, if there exists
a cyclic $m$-cycle system of $K_{mt}-I$, then
\begin{enumerate}
\item $t \equiv 0,2 \ (\mod 4)$ when $m \equiv 0\ (\mod 8),$
\item $t \equiv 0,1 \ (\mod 4)$ when $m \equiv 2\ (\mod 8)$ with
$t > 1$ if $m = 2p^\alpha$ for some prime $p$ and integer $\alpha
\ge 1,$
\item $t \equiv 0,3 \ (\mod 4)$ when $m \equiv 6\ (\mod 8)$.
\end{enumerate}
\end{corollary}

Let $n > 0$ be an integer and
suppose there exists an ordered $m$-tuple $(d_1,d_2,\dots,d_m)$
satisfying each of the following:
\begin{enumerate}
\item[(i)]
$d_i$ is an integer for $i=1,2,\dots,m$;
\item[(ii)]
$|d_i| \neq|d_j|$ for $1\leq i<j\leq m$;
\item[(iii)]
$d_1+d_2+\dots+d_m\equiv 0(\mod n)$; and
\item[(iv)]
$d_1+d_2+\dots+d_r\not\equiv d_1+d_2+\dots+d_s(\mod n)$ for $1\leq
r<s\leq m$.
\end{enumerate}
Then an $m$-cycle $C$ can be constructed from this $m$-tuple, that
is, let $C = (0,d_1,d_1+d_2,\dots,d_1+d_2+\dots+d_{m-1})$, and
$\{C\}$ is a minimum generating set for a cyclic $m$-cycle system of
$\langle \{d_1, d_2 ,\dots, d_m\} \rangle_n$.  Thus, in what follows, to
find cyclic $m$-cycle systems of $\langle L \rangle_n$, it suffices to
partition $L$ into $m$-tuples satisfying the above conditions.
Hence, an $m$-tuple satisfying (i)-(iv) above is called a {\it
difference $m$-tuple} and it {\it corresponds} to the $m$-cycle
$C=(0, d_1, d_1+d_2, \ldots, d_1+d_2+\cdots d_{m-1})$ in $ \langle
L\rangle_n$.

\section{The Case when $m \equiv 0 \pmod 8$}

In this section, we consider the case when $m \equiv 0\ (\mod 8)$
and show that there exists a cyclic $m$-cycle system of $K_{mt}-I$
for each even positive integer $t$. We begin with the case $t=2$.

\begin{lemma} \label{0mod8_2m}
For each positive integer $m \equiv 0 \pmod{8}$, there exists a cyclic $m$-cycle system of $K_{2m}-I$.
\end{lemma}

\begin{proof}
Let $m$ be a positive integer such that $m \equiv 0\ (\mod
8)$, say $m=8r$ for some positive integer $r$.  Then $K_{2m}
-I = \langle S' \rangle_{2m}$ where $S' = \{1, 2, \ldots, m-1\}= \{1, 2, \ldots, 8r-1\}$. The proof proceeds as follows.  We begin by finding a path $P$ of length $m/2=4r$, ending at vertex $m$, so that $C = P \cup \rho^m(P)$ is an $m$-cycle.  Note that $\langle \{2\}\rangle_{2m}$ consists of two vertex disjoint $m$-cycles. For the remaining $4r-2$ edge lengths in $S' \setminus (\ell(P)\cup \{2\})$, we find $2r-1$ paths $P_i$ of length $2$, ending at vertex $4$ or $-4$, so that $C_i = P_i \cup \rho^4(P_i) \cup \rho^8(P_i) \cup \cdots \rho^{2m-4}(P_i)$ is an $m$-cycle.  Then this collection of cycles will give a minimum generating set for a cyclic $m$-cycle system of $K_{2m}-I$.

Suppose first that $r$ is odd. For $r = 1$, let $P: 0, -3, 3, 7, 8$ and note that the edge lengths of $P$ in the order encountered are $3, 6, 4, 1$. For $r = 3$, let $$P: 0, -3, 3, -7, 7, -11, 11, 23, 19, 20, -20, -4, 24$$ and note that edge lengths of $P$ in the order encountered are $3, 6, 10, 14, 18, 22, 12, 4, 1, 8, 16, 20$.  For $r \ge 5$, let
\begin{eqnarray*}
P&:& 0, -3, 3, -7, 7, \ldots, -(4r-1), 4r-1, 8r-1, 8r-5, 8r-4, 8r+4, 8r-8, 8r+8, \ldots, \\
&&6r-2, 10r-2, 6r-10, 10r+2, 6r-14, \ldots, 12r-8, 4r-4, 8r
\end{eqnarray*}
be a path of length $m/2$ whose edge lengths in the order encountered are
$3,6,10,14,\ldots,8r-6,8r-2,4r,4,1,8,12,16,\ldots,4r-4,4r+8,4r+12,\ldots,8r-8,8r-4,4r+4$.

Now suppose that $r$ is even. For $r = 2$, let $P: 0, -3, 3, -7, 7, -1, -5, -4, 16$ and note that the edge lengths of $P$ in the order encountered are $3, 6, 10, 14, 8, 4, 1, 12$.  For $r \ge 4$, let
\begin{eqnarray*}
P&:& 0, -3, 3, -7, 7, \ldots, -(4r-1), 4r-1, -1, -5, -4, 4, -8, 8, \ldots, \\
&&-(2r-4), 2r-4, -2r, 2r+8, -(2r+4), 2r+12, \ldots, -(4r-8), 4r, -(4r-4), 8r
\end{eqnarray*}
 be a path of length $m/2$ whose edge lengths in the order encountered are
$3,6,10,14,\ldots,8r-6,8r-2,4r,4,1,8,12,16,\ldots,4r-8,4r-4,4r+8,4r+12,\ldots,8r-8,8r-4,4r+4$.

 In each case, let $C = P \cup \rho^m(P)$ and observe that $C$ is an $m$-cycle $C$ with $\ell(C)=\{1,3,4,6,8,\ldots, 8r-2\}$. Let $C' = (0, 2, 4, 6, \ldots, 2m-2)$ and note that $C'$ is an $m$-cycle with $\ell(C') = \{2\}$.

For $0 \le i \le r-2$, let $P_i:0, 9+8i, 4$ be the path of length 2 with edge lengths $9+8i, 5+8i$ and let $P_i': 0, 11+8i, 4$ be the path of length $2$ with edge lengths $11+8i,7+8i$.  Let $C_i = P_i \cup \rho^4(P_i) \cup \rho^8(P_i) \cup \cdots \rho^{2m-4}(P_i)$  and $C_i' = P_i' \cup \rho^4(P_i') \cup \rho^8(P_i') \cup \cdots \rho^{2m-4}(P_i')$  and note that each is an $m$-cycle with $\ell(C_i)=\{5+8i,9+8i\}$ and $\ell(C_i')=\{7+8i,11+8i\}$.

Finally, let $P'': 0, 8r-3, -4$ be the path of length $2$ with edge lengths $8r-3$ and $8r-1$. Let $C'' = P'' \cup \rho^4(P'') \cup \rho^8(P'') \cup \cdots \rho^{2m-4}(P'')$ and note that $C''$ is an $m$-cycle with $\ell(C'') = \{8r-3, 8r-1\}$.

Then $\{C,C', C_0, \ldots, C_{r-2}, C_0', \ldots, C_{r-2}', C''\}$ is a minimum generating set for a cyclic $m$-cycle system of $K_{2m}-I$.
\end{proof}

We now consider the case when $t$ is even and $t > 2$.

\begin{lemma}
For each positive integer $k$ and each positive integer $m \equiv 0\ (\mod 8)$, there
exists a cyclic $m$-cycle system of $K_{2mk} - I$.
\end{lemma}

\begin{proof}
Let $m$ and $k$ be positive integers such that $m \equiv 0\ (\mod
8)$. Lemma \ref{0mod8_2m} handles the case when $k = 1$ and thus we may assume that $k \ge 2$. Then $K_{2km}
-I = \langle S' \rangle_{2km}$ where $S' = \{1, 2, \ldots, km-1\}$.
Since $K_{2m} - I$ has a cyclic $m$-cycle system by Lemma \ref{0mod8_2m} and
$\langle \{k, 2k, \ldots, mk\}\rangle_{2km}$ consists of
$k$ vertex-disjoint copies of $K_{2m} - I$, we need only show that
$\langle S\rangle_{2km}$ has a cyclic $m$-cycle system where
$S = \{1, 2, \ldots, mk\} \setminus \{k, 2k, \ldots,
mk\}$.

Let $A = [a_{i,j}]$ be the $(k-1)
\times m$  array
\begin{eqnarray*} {\footnotesize
\left[ \begin{array}{lllllllll}
   k-1    & 2k-1   & 3k-1    & 4k-1   &        &   (m-1)k-1 & mk-1 \\
   \vdots & \vdots & \vdots  & \vdots & \cdots &   \vdots   & \vdots \\
   2      & k+2    & 2k+2    & 3k+2   &        &  (m-2)k+2 & (m-1)k+2\\
   1      & k+1    & 2k+1    & 3k+1   &        &   (m-2)k+1 & (m-1)k+1 \\
                     \end{array}
              \right].}
\end{eqnarray*}

It is straightforward to verify that $A$ satisfies
\[ \sum_{j \equiv 0,1\ (\mod 4)} a_{i,j} = \sum_{j \equiv 2,3\ (\mod 4)}
a_{i,j},\]
and \[a_{i,1} < a_{i,2} < \ldots < a_{i, m}\] for  each $i$ with
$1\le i \le k$.

For each $i = 1, 2, \ldots, k-1$, the $m$-tuple
$$(a_{i,1}, -a_{i,3}, a_{i,5}, -a_{i,7}, \ldots, a_{i,m-3},
-a_{i,m-1}, -a_{i,m-2}, a_{i,m-4}, -a_{i,m-6}, \ldots, -a_{i,6},
a_{i,4}, -a_{i,2}, a_{i,m})$$ is a difference $m$-tuple and
corresponds to an $m$-cycle $C_i$ with $\ell(C_i) = \{a_{i,1},
a_{i,2}, \ldots,
a_{i,m}\}$.  Hence, $X = \{C_1, C_2, \ldots,
C_{k-1}\}$ is a minimum generating set for a cyclic $m$-cycle system of
$\langle S\rangle_{2km}$.
\end{proof}

\section{The Case when $m \equiv 4 \ (\mod 8)$}

In this section, we consider the case when $m \equiv 4\ (\mod 8)$
and show that there exists a cyclic $m$-cycle system of $K_{mt}-I$
for each $t \ge 1$. We begin with the case when $t$ is odd, say $t = 2k+1$ for some nonnegative integer $k$.

\begin{lemma} \label{k-odd}
For each nonnegative integer $k$ and each $m \equiv 4\ (\mod 8)$, there
exists a cyclic $m$-cycle system of $K_{m(2k+1)} - I$.
\end{lemma}

\noindent
\begin{proof}
Let $m$ and $k$ be nonnegative integers such that $m \equiv 4\ (\mod
8)$.  Since $K_m - I$ has a cyclic hamiltonian cycle system \cite {GM}, we may assume that $k \ge 1$. Let $m = 4r$ for some positive integer $r$. Then $K_{m(2k+1)}
-I = \langle S'\rangle_{(2k+1)m}$ where $S' = \{1, 2, \ldots, 4rk+2r-1\}$.
Again, since $K_m - I$ has a cyclic hamiltonian cycle system \cite {GM} and
$\langle \{2k+1, 4k+2, \ldots, (2r-1)(2k+1)\}\rangle_{(2k+1)m}$ consists of
$2k+1$ vertex-disjoint copies of $K_m - I$, we need only show that
$\langle S\rangle_{(2k+1)m}$ has a cyclic $m$-cycle system where
$$S = \{1, 2, \ldots, 4rk+2r-1\} \setminus \{2k+1, 4k+2, \ldots,
(2r-1)(2k+1)\}.$$

Let $r$ and $k$ be positive integers.  Let $A = [a_{i,j}]$ be the $k
\times m$  array
\begin{eqnarray*} {\footnotesize
\left[ \begin{array}{lllllllll}
                      k & 2k & 3k+1 & 4k+1 & 5k+2 & & (4r-2)k+2r -2 &
                      (4r-1)k+2r-1 & 4rk+2r-1 \\
                      \vdots & \vdots & \vdots & \vdots & \vdots & \cdots &
\vdots & \vdots & \vdots \\
                2 & k+2 & 2k+3 & 3k+3 & 4k + 4 & & (4r-3)k+2r &
(4r-2)k+2r+1 & (4r-1)k+2r+1\\
                1 & k+1 & 2k+2 & 3k+2 &4k+3 & & (4r-3)k+2r-1 & (4r-2)k+2r &
(4r-1)k+2r \\
                     \end{array}
              \right].}
\end{eqnarray*}

It is straightforward to verify that $A$ satisfies
\[ \sum_{j \equiv 0,1\ (\mod 4)} a_{i,j} = \sum_{j \equiv 2,3\ (\mod 4)}
a_{i,j},\]
and \[a_{i,1} < a_{i,2} < \ldots < a_{i, m}\] for  each $i$ with
$1\le i \le k$.

For each $i = 1, 2, \ldots, k$, the $m$-tuple
$$(a_{i,1}, -a_{i,3}, a_{i,5}, -a_{i,7}, \ldots, a_{i,m-3},
-a_{i,m-1}, -a_{i,m-2}, a_{i,m-4}, -a_{i,m-6}, \ldots, -a_{i,6},
a_{i,4}, -a_{i,2}, a_{i,m})$$ is a difference $m$-tuple and
corresponds to an $m$-cycle $C_i$ with $\ell(C_i) = \{a_{i,1},
a_{i,2}, \ldots,
a_{i,m}\}$.  Hence, $X = \{C_1, C_2, \ldots,
C_k\}$ is a minimum generating set for a cyclic $m$-cycle system of
$K_{m(2k+1)}-I$.

\end{proof}

We now handle the case when $t$ is even, say $t = 2k$ for some positive integer $k$.

\begin{lemma}
For each positive integer $k$ and each $m \equiv 4\ (\mod 8)$, there
exists a cyclic $m$-cycle system of $K_{2mk} - I$.
\end{lemma}

\noindent
\begin{proof}
As before, let $m$ and $k$ be positive integers such that $m \equiv
4\ (\mod 8)$. Thus $m = 4r$ for some positive integer $r$. Then
$K_{2mk} -I = \langle S' \rangle_{2km}$ where $S' = \{1, 2, \ldots, 4rk-1\}$.
Since $K_m - I$ has a cyclic hamiltonian cycle system \cite {GM} and
$\langle \{2k, 4k, \ldots, (2r-1)(2k)\}\rangle_{2km}$ consists of $2k$
vertex-disjoint copies of $K_m - I$, we need only show that $\langle S \rangle_{2km}$ has a cyclic $m$-cycle system where
$$S = \{1, 2, \ldots, 4rk-1\} \setminus \{2k, 4k, \ldots,
(2r-1)(2k)\}.$$  Since $|S|= m(k-1) + m/2$, we will start by
partitioning a subset $T \subseteq S$ with $|T|=m(k-1)$ into $k-1$
difference $m$-tuples.

Let $T = \{1, 2, \ldots,
4rk-1\} \setminus \{1, 2k, 4k-1, 4k, 4k+1, 6k, 8k-1, 8k, 8k +1,
\ldots, (4r-4)k-1, (4r-4)k, (4r-4)k+1, (4r-2)k, 4rk-1\},$ and
observe that $|T|=(k-1)m$. Let $A = [a_{i,j}]$, with entries from
the set $T$, be the $(k-1) \times m$ array
\begin{eqnarray*} {\footnotesize
\left[ \begin{array}{llllllllllllll}
                k & 2k-1   & 3k-1 & 4k-2   & 5k     &6k-1    & 7k-1 & 8k-2
& 9k & \\
           \vdots & \vdots&\vdots & \vdots & \vdots & \vdots & \vdots &
\vdots & \vdots & \cdots \\
                3 & k+2 &  2k+2 & 3k+1     & 4k + 3 &5k+2    &6k+2 & 7k+1 &
8k+3 & \\
                2 & k+1 &  2k+1 & 3k       &4k+2    & 5k+1   & 6k+1 & 7k &
8k+2 & \\
                     \end{array}
              \right. }
\end{eqnarray*}
\begin{eqnarray*} {\footnotesize
\left. \begin{array}{llllllllllllll}
                      & (4r-3)k &(4r-2)k-1 & (4r-1)k-1 & 4rk-2 \\
                      \cdots & \vdots & \vdots & \vdots & \vdots \\
                      & (4r-4)k+3 & (4r-3)k+2 & (4r-2)k+2 & (4r-1)k+1\\
                      & (4r-4)k+2 &(4r-3)k+1 & (4r-2)k+1 & (4r-1)k \\
                     \end{array}
              \right].}
\end{eqnarray*}
It is straightforward to verify that the array $A$ satisfies
\[ \sum_{j \equiv 0,1\ (\mod 4)} a_{i,j} = \sum_{j \equiv 2,3\ (\mod 4)}
a_{i,j},\]
and \[a_{i,1} < a_{i,2} < \ldots < a_{i, m}\] for  each $i$ with
$1\le i \le k-1$.

For each $i = 1, 2, \ldots, k-1$, the $m$-tuple
$$(a_{i,1}, -a_{i,3}, a_{i,5}, -a_{i,7}, \ldots, a_{i,m-3},
-a_{i,m-1}, -a_{i,m-2}, a_{i,m-4}, -a_{i,m-6}, \ldots, -a_{i,6},
a_{i,4}, -a_{i,2}, a_{i,m})$$ is a difference $m$-tuple and
corresponds to an $m$-cycle $C_i$ with $\ell(C_i) = \{a_{i,1},
a_{i,2}, \ldots,
a_{i,m}\}$.  Hence, $X = \{C_1, C_2, \ldots,
C_{k-1}\}$ is a minimum generating set for a cyclic $m$-cycle system of
$\langle T\rangle_{2km}$.

It now remains to find a minimum generating set for a cyclic $m$-cycle
system of $\langle B \rangle_{2km}$ where $B= \{1, 4k-1, 4k+1, 8k-1, 8k+1,
\ldots, (4r-4)k-1, (4r-4)k+1, 4rk-1\}$.  For $i = 1, 2, \ldots, r$,
define $d_{2i-1} = 4(i-1)k+1$ and $d_{2i}=4ik-1$. Observe that $B =
\{d_1, d_2, \ldots, d_{2r}\}$ and $d_{j+2} - d_j = 4k$ for $j = 1,
2, \ldots, 2r-2$. Suppose first that $r$ is even.  For $i = 1, 2,
\ldots, r$, let $P_i: 0, d_{2i+1}, 4k$ if $i$ is odd and let $P_i:
0, d_{2i}, 4k$ if $i$ is even. Let $C_i' = P_i \cup \rho^{4k}(P_i)
\cup \rho^{8k}(P_i) \cup \cdots \cup \rho^{(2m-4)k}(P_i)$, and note
that $C_i'$ is an $m$-cycle with $\ell(C_i') = \{d_{2i-1},
d_{2i+1}\}$ if $i$ is odd or $\ell(C_i') = \{d_{2i-2}, d_{2i}\}$ if
$i$ is even. Now suppose $r$ is odd. Let $P_1: 0, 1, 4k$, and let
$P_i: 0, d_{2i+1}, 4k$ if $i$ is even and let $P_i: 0, d_{2i}, 4k$
if $i$ is odd.  Define $C_i'$ as in the case when $r$ is even and
note that $\ell(C_1') =\{1, 4k-1\}$, $\ell(C_i')= \{d_{2i-1},
d_{2i+1}\}$ if $i$ is even, and $\ell(C_i')=\{d_{2i-2}, d_{2i}\}$ if
$i$ is odd.  In either case, $\ell(C_1') \cup \ell(C_2') \cup \cdots
\cup  \ell(C_r') = B$ so that $\{C_1', C_2', \ldots, C_r'\}$ is a
minimum generating set for $\langle B \rangle_{2km}$.
\end{proof}

\section{The Case when $m \equiv 2 \ (\mod 4)$}

In this section, we consider the case when $m \equiv 2\ (\mod 4)$
and prove parts (2) and (4) of Theorem~1.1.  We divide this proof into three parts, each dealt with in its own
subsection. First we consider the case $t\equiv 0\ (\mod 4)$. Then we consider the case $m \equiv 2\ (\mod 8)$ and $t \equiv 1\ (\mod 4)$. Finally we consider the case $m \equiv 6\ (\mod 8)$ and $t \equiv 3\ (\mod 4)$.

\subsection{The case when $t \equiv 0\ (\mod 4)$.} \hspace{1in}

\vspace*{12pt}
We consider the case $t \equiv 0\ (\mod 4)$, starting with the special case $t = 4$.

\begin{lemma} \label{4m}
For each positive integer $m\ge 6$ with $m \equiv 2 \pmod{4}$, there exists a cyclic $m$-cycle system of $K_{4m}-I$.
\end{lemma}

\begin{proof} Let $m \ge 6$ be a positive integer with $m \equiv 2\ (\mod 4)$. Then $K_{4m}
-I = \langle  S' \rangle_{4m}$ where $S' = \{1, 2, \ldots, 2m-1\}$. The proof proceeds as follows.  We begin by finding one difference $m$-tuple which corresponds to an $m$-cycle $C$ with $|\ell(C)|= m$. Note that $\langle\{4\} \rangle_{4m}$ consists of four vertex disjoint $m$-cycles. For the remaining $m-2$ edge lengths in $S' \setminus (\ell(C)\cup \{4\})$, we find $(m-2)/2$ paths $P_i$ of length $2$, ending at vertex $8$ or $-8$, so that $C_i = P_i \cup \rho^8(P_i) \cup \rho^{16}(P_i) \cup \cdots \rho^{4m-8}(P_i)$ is an $m$-cycle.  Then this collection of cycles will give a minimum generating set for a cyclic $m$-cycle system of $K_{2m}-I$.

Consider the difference $m$-tuple
$$(1,-2,6,-10,\ldots,2m-6,-(2m-2), -3,8,-12,\ldots,2m-12,-(2m-8),2m-4)$$
and the corresponding $m$-cycle $C$ with $\ell(C)=\{1,2,3,6,8,\ldots, 2m-2\}$.  It is straightforward to verify that the odd vertices visited all lie between $-m+1$ and $m-1$ with no duplication.  Similarly, the even vertices visited all lie between $-2m+4$ and $-4$, and have no duplication.

Let $C' = (0, 4, 8, \ldots, 4m-4)$ and note that $C'$ is an $m$-cycle with $\ell(C') = \{4\}$.

Let $m=8k+m'$, so $m'$ is either $2$ or $6$. If $k=0$, then $m' = 6$ and let $P: 0, 13, 8$ be the path of length $2$ with edge lengths  $11, 5$. Then, $C'' = P \cup \rho^8(P) \cup \rho^{16}(P)$ is a $6$-cycle with $\ell(C'') = \{11, 5\}$. Then $\{C, C', C''\}$ is a minimum generating set for cyclic $6$-cycle system of $K_{24}-I$.  Now suppose that $k \ge 1$.  For $0 \le i \le k-1$, let $P_i: 0, 13+16i, 8$ be the path of length $2$ with edge lengths  $13+16i, 5+16i$; let $P_i': 0, 15+16i, 8$ be the path of length $2$ with edge lengths  $15+16i,7+16i$; let $P_i'': 0, 17+16i, 8$ be the path of length $2$ with edge lengths $17+16i,9+16i$; and let $P_i''':0, 19+16i, 8$ with edge lengths $19+16i,11+16i$.  Let $C_i = P_i \cup \rho^8(P_i) \cup \rho^{16}(P_i) \cup \cdots \rho^{4m-8}(P_i)$, $C_i' = P_i' \cup \rho^8(P_i') \cup \rho^{16}(P_i') \cup \cdots \rho^{4m-8}(P_i')$, $C_i'' = P_i'' \cup \rho^8(P_i'') \cup \rho^{16}(P_i'') \cup \cdots \rho^{4m-8}(P_i'')$, and $C_i''' = P_i''' \cup \rho^8(P_i''') \cup \rho^{16}(P_i''') \cup \cdots \rho^{4m-8}(P_i''')$  and note that each is an $m$-cycle with $\ell(C_i)=\{5+16i,13+16i\}$, $\ell(C_i')=\{7+16i,15+16i\}$, $\ell(C_i'')=\{9+16i,17+16i\}$, and $\ell(C_i''')=\{11+16i,19+16i\}$.

If $m'=2$, then $\{C,C', C_0,C_0', C_0'', C_0''', \ldots, C_{k-1},C_{k-1}',C_{k-1}'', C_{k-1}'''\}$ is a minimum generating set for a cyclic $m$-cycle system of $K_{4m}-I$ .  If $m'=6$, then let $P_k: 0, 2m-1, -8$ and $P_k': 0, 2m-3, -8$ be paths of length $2$ with $\ell(P_k)= \{2m-1, 2m-7\}$ and $\ell(P_k')= \{2m-3, 2m-5\}$.  Let $C_k = P_k \cup \rho^8(P_k) \cup \rho^{16}(P_k) \cup \cdots \rho^{4m-8}(P_k)$ and $C_k' = P_k' \cup \rho^8(P_k') \cup \rho^{16}(P_k') \cup \cdots \rho^{4m-8}(P_k')$ and observe that each is an $m$-cycle with $\ell(C_k) = \{2m-1, 2m-7\}$ and $\ell(C_k')= \{2m-3, 2m-5\}$.  Thus, $\{C,C', C_0,C_0', C_0'', C_0''', \ldots, C_{k-1},C_{k-1}',C_{k-1}'', C_{k-1}''', C_k, C_k'\}$  is a minimum generating set for a cyclic $m$-cycle system of $K_{4m}-I$.
\end{proof}

We now consider the case when $t \equiv 0\ (\mod 4)$ with $t > 4$.

\begin{lemma}
For each positive integer $k$ and each positive integer $m \equiv 2\ (\mod 4)$ with $m \ge 6$, there
exists a cyclic $m$-cycle system of $K_{4mk} - I$.
\end{lemma}

\begin{proof}
Let $m\ge 6$ and $k$ be positive integers such that $m \equiv 2\ (\mod
4)$. Lemma \ref{4m} handles the case when $k = 1$ and thus we may assume that $k \ge 2$. Then $K_{4km}
-I = \langle S' \rangle_{4km}$ where $S' = \{1, 2, \ldots, 2km-1\}$.
Since $K_{4m} - I$ has a cyclic $m$-cycle system by Lemma \ref{4m} and
$\langle \{k, 2k, \ldots, 2km\}\rangle_{4km}$ consists of
$k$ vertex-disjoint copies of $K_{4m} - I$, we need only show that
$\langle S\rangle_{2km}$ has a cyclic $m$-cycle system where
$S = \{1, 2, \ldots, 2km\} \setminus \{k, 2k, \ldots,
2km\}$.

Let $A = [a_{i,j}]$ be the $2k
\times m$  array
\begin{eqnarray*} {\footnotesize
\left[ \begin{array}{lllllllll}
   2k    & 4k   & 6k    & 8k   &        &   (m-1)2k & 2km \\
   2k-1 & 2k+1 & 6k-1 & 8k-1 &          & (m-1)2k-1 & 2km - 1 \\
   \vdots & \vdots & \vdots  & \vdots & \cdots &   \vdots   & \vdots \\
   2      & 4k-2   & 4k+2    & 6k+2   &        &  (m-2)2k+2 & (m-1)2k+2\\
   1      & 4k-1    & 4k+1    & 6k+1   &        &   (m-2)2k+1 & (m-1)2k+1 \\
                     \end{array}
              \right].}
\end{eqnarray*}
(Observe that the second column does not follow the same pattern as the others.)

Let $A'$ be the $(2k-2) \times m$ array obtained from $A$ by deleting rows $1$ and $k+1$. Then the entries in $A'$ are precisely the elements of $S$. Also, it is straightforward to verify that $A'$ satisfies
\[ a_{i,j} + a_{i, j+3} = a_{i, j+1} + a_{i, j+2}\]
for each positive integer $j \equiv 3\ (\mod 4)$ with $j \le m-3$, $$a_{i,1} + a_{i,2} + a_{i, m-3} + a_{i,m-1} = a_{i, m-2} + a_{i, m},$$
and \[a_{i,1} < a_{i,2} < \ldots < a_{i, m}\] for  each $i$ with
$1\le i \le 2k-2$.

For each $i = 1, 2, \ldots, 2k-2$, the $m$-tuple
$$(a_{i,1}, a_{i, 2},  -a_{i, 4}, a_{i,6}, -a_{i,8}, a_{i,10}, \ldots, -a_{i,m-2},
-a_{i,m}, a_{i,m-3}, -a_{i,m-5}, a_{i,m-7}, \ldots, a_{i,3}, a_{i,m-1})$$ is a difference $m$-tuple and
corresponds to an $m$-cycle $C_i$ with $\ell(C_i) = \{a_{i,1},
a_{i,2}, \ldots,
a_{i,m}\}$.  Hence, $X = \{C_1, C_2, \ldots,
C_{2k-2}\}$ is a minimum generating set for a cyclic $m$-cycle system of
$\langle S\rangle_{4km}$.
\end{proof}

What remains is to find cyclic $m$-cycle systems of $K_{mt}-I$ for the appropriate odd values of $t$, which we do in the following subsections.

\subsection{The case when $m \equiv 2\ (\mod 8)$ and $t \equiv 1\ (\mod 4)$.} \hspace{1in}

\vspace*{12pt}
In this subsection, we find a cyclic $m$-cycle system of $K_{mt}-I$ when $m \equiv 2\ (\mod 8)$ and $t \equiv 1\ (\mod 4)$.  We begin with two special cases, namely when $m = 10$ or $t = 5$.

\begin{lemma} \label{m=10}
For each positive integer $t\equiv 1\ (\mod 4)$ with $t > 1$, there exists a cyclic $10$-cycle system
of $K_{10t} - I$.
\end{lemma}

\begin{proof}
Let $t\equiv 1\ (\mod 4)$ with $t > 1$, say $t = 4s + 1$ where $s \ge 1$.  Then $K_{10t} -I = \langle S' \rangle_{10t}$ where $S' = \{1, 2,
\ldots, 20s+4\}$. Consider the paths $P_{1}
: 0, 5t-1, 2t$ and $P_2: 0, 5t-2, 2t$.  Then, $\ell(P_1) = \{3t-1,
5t-1\}$ and $\ell(P_2) = \{3t-2, 5t-2\}$.  For $i\in \{1, 2\}$, let
$C_i = P_{i} \cup \rho^{2t}(P_{i}) \cup \rho^{4t}(P_{i}) \cup \cdots
\cup \rho^{8t}(P_{i})$.  Then clearly each $C_i$ is an $10$-cycle
and $X = \{C_1, C_2\}$ is a minimum generating set for $ \langle \{3t-2, 3t-1, 5t-2, 5t-1\}\rangle_{10t}$.  Since $3t-3 = 12s$ and $5t-2 = 20s+3$, it remains to
find a cyclic 10-cycle system of $\langle S \rangle_{10t}$ where $S = \{1, 2,
\ldots, 12s, 12s+3, 12s+4, \ldots, 20s+2\}$.  Let $A = [a_{i,j}]$ be
the $2s \times 10$ array
\begin{eqnarray*} {\footnotesize
\left[ \begin{array}{lllllllllllllll}
1 & 2 & 3 & 4 & 8s+1 & 8s+3 & 12s+3 & 12s+4 & 12s+5 & 12s+6 \\
5 & 6 & 7 & 8 & 8s+2 & 8s+4 & 12s+7 & 12s+8 & 12s+9 & 12s+10 \\
\vdots & \vdots & \vdots & \vdots & \vdots & \vdots & \vdots & \vdots\\
8s-3 & 8s-2 & 8s-1 & 8s &12s-2 & 12s & 20s-1 & 20s & 20s+1 & 20s + 2\\
                     \end{array}
              \right]}
\end{eqnarray*}
Clearly, for each $i$ with $1\le i \le 2s$,
\[ a_{i,2} + \sum_{j \equiv 0,1\ (\mod 4)} a_{i,j} = a_{i,1} + \sum_{j
\equiv 2,3\
(\mod 4)} a_{i,j}\ (\mbox{where } 3 \le j \le 10)\]
and \[a_{i,1} < a_{i,2} < \ldots < a_{i, 10}.\] Thus the $10$-tuple
$$(a_{i,1}, -a_{i,2}, a_{i,3}, -a_{i,5}, a_{i,7},
-a_{i,9}, -a_{i,8},  a_{i,6}, -a_{i,4}, a_{i,10})$$ is a difference
$10$-tuple and corresponds to an $10$-cycle $C'_i$ with $\ell(C'_i)
= \{a_{i,1}, a_{i,2}, \ldots,
a_{i,10}\}$.  Hence, $X' = \{C'_1, C'_2, \ldots,
C'_{2s}\}$ is a minimum generating set for a cyclic $10$-cycle system
of $ \langle S' \setminus B \rangle_{10t}$.
\end{proof}

We now consider the case when $t = 5$.

\begin{lemma} \label{t=5}
For each positive integer $m \equiv
2\ (\mod 8)$, there exists a cyclic $m$-cycle system
of $K_{5m} - I$.
\end{lemma}

\begin{proof}
Let $m$ be a positive integer such that $m \equiv 2\ (\mod
8)$, say $m = 8r+2$ for some positive
integer $r$. By Lemma \ref{m=10}, we may assume $r \ge 2$. Then $K_{5m} -I = \langle S' \rangle_{5m}$ where $S' = \{1, 2,
\ldots, 20r+4\}$.

 Let $2r = 6q + 4 + b$ for
integers $q \ge 0$ and $b \in \{0,2,4\}$. Let $a$ be a positive
integer such that $1 + \log_2(q+2) \le a \le 1 + \log_2(5q+2),$ and
note that $a$ exists since if $q=0$ then $\log_2(q+2)$ is an integer, while if $q\ge 1$ then $2(q+2)=2q+4\le 4q+2<5q+2$. For nonnegative integers $i$ and $j$,
define $d_{i,j} = 10(2r-i) + j$. Consider the path $P_{i,j} : 0,
d_{i,j}, 5 \cdot 2^a$ and observe that $\ell(P_{i,j}) = \{10(2r-i) +
j, 10(2r-i) + j- 5 \cdot 2^a\}$.  If $0< j < 10$, then $C_{i,j} =
P_{i,j} \cup \rho^{10}(P_{i,j}) \cup \rho^{20}(P_{i,j}) \cup \cdots
\cup \rho^{5m-10}(P_{i,j})$  is an $m$-cycle since $m \equiv 2\
(\mod 8)$ gives $\gcd(5 \cdot 2^a, 5m) = 10$.  Thus, if $0 < j < 10$,
$\ell(C_{i,j}) = \{10(2r-i) + j, 10(2r-i) + j- 5 \cdot 2^a\}$. Let
$$X = \{C_{0,j} \mid  1 \le j \le 4\} \cup \{C_{i,j} \mid 1 \le i \le q
\mbox{ and } 1 \le j \le 6\} \cup
\{C_{q+1,j} \mid  6-b+1 \le j \le 6\}$$ and let
\begin{eqnarray*}
B &=& \{20r+j, 20r+j - 5 \cdot 2^a \mid  1 \le j \le 4\} \\ && \cup
\ \{10(2r-i)+j, 10(2r-i)+j - 5 \cdot 2^a \mid 1 \le i \le q \mbox{
and } 1 \le j \le 6\} \\
&&\cup \ \{10(2r-q-1)+j, 10(2r-q-1)+j-5 \cdot 2^a \mid  6-b+1 \le j
\le 6\},
\end{eqnarray*}
where if $q=0$ or $b=0$, we take the corresponding sets to be empty
as necessary. Now $B$ will consist of $4r$ distinct lengths and $X$
will be a minimum generating set for $\langle B \rangle_{5m}$ if $20r+4 - 5 \cdot
2^a \le 10(2r-q-1) + 6-b$. Note that $1 + \log_2(q+2) \le a \le 1 +
\log_2(5q+2)$ gives $q+2 \le 2^{a-1} \le 5q+2$.  So,
$$20r+4 - [10(2r-q-1) + 6-b] = 10q+8+b \le 10q+12$$ and $$(10q+12)/10
< q+2 \le 2^{a-1}.$$ Thus $20r+4 - 5 \cdot 2^a \le 10(2r-q-1) +
6-b$ so that $B$ consists of $4r$ distinct lengths, and $X$ is a minimum
generating set for $\langle B \rangle_{5m}$.

It remains to find a cyclic $m$-cycle system of $\langle S'
\setminus B \rangle_{5m}$. The smallest length in $B$ is $10(2r-q-1) +
6-b+1-5\cdot 2^a$, and we wish to show $10(2r-q-1) + 6-b - 5 \cdot
2^a \ge 12$.  So, $$10(2r-q-1) + 6-b - 12 = 20r - 10q - 16 -b \ge
20r - 10q - 20$$ and $(20r-10q-20)/10 \ge 2r-q-2$.  Now $$2r-q-2 =
5q+ 2 + b \ge 5q+2 \ge 2^{a-1}.$$ Hence, $10(2r-q-1) + 6-b - 5 \cdot
2^a \ge 12$. Since $|B| = 4r$, we have $|S' \setminus B| = 20r+4 -4r
= 2(8r+2)$. Now
\begin{eqnarray*}
S' \setminus B &=& \{1, 2, \ldots, 10(2r-q-1) + 6-b - 5 \cdot
2^a\}  \cup \  \{10(2r-i) -5\cdot 2^a - 3,  \\
&& 10(2r-i) -5\cdot 2^a - 2, 10(2r-i) -5\cdot 2^a - 1,
10(2r-i) -5\cdot 2^a \mid 0 \le i \le q\}\\
&& \cup \ \{10(2r) + 5 - 5 \cdot 2^a, \ldots, 10(2r-q-1) +
6-b\} \\
&& \cup \ \{10(2r-i) - 3, 10(2r-i) - 2, 10(2r-i) - 1,
10(2r-i) \mid 0 \le i \le q\}.
\end{eqnarray*}
Note that each the sets $\{1, 2, \ldots, 10(2r-q-1) + 6-b - 5 \cdot
2^a\}, \{10(2r-i) -5\cdot 2^a - 3, 10(2r-i) -5\cdot 2^a - 2,
10(2r-i) -5\cdot 2^a - 1, 10(2r-i) -5\cdot 2^a \mid 0 \le i \le q\},
\{10(2r) + 5 - 5 \cdot 2^a, \ldots, 10(2r-q-1) + 6-b\}$, and
$\{10(2r-i) - 3, 10(2r-i) - 2, 10(2r-i) - 1, 10(2r-i) \mid 0 \le i
\le q\}$ has even cardinality and consists of consecutive integers.
Therefore, we may partition $S' \setminus B$ into sets $T, S_1, S_2,
\ldots, S_{8r-4}$ where $T = \{1, 2, \ldots, 12\}$ and for $i = 1,
2, \ldots, 8r-4$, let $S_i = \{b_{i}, b_{i}+1\}$ with $b_{1} < b_{2}
< \cdots < b_{8r-4}.$

Let $A =
[a_{i,j}]$ be the $2 \times m$ array
\begin{eqnarray*} {\footnotesize
\left[ \begin{array}{llllllllllllll}
1 & 2 & 3 & 4 & 9 & 11 & b_1 & b_1+1 & b_2 & b_2+1 & \cdots & b_{4r-2} &
b_{4r-2} + 1\\
5 & 6 & 7 & 8 & 10 & 12 & b_{4r-1} & b_{4r-1}+1 & b_{4r} & b_{4r}+1 & \cdots
& b_{8r-4} & b_{8r-4}+1 \\
                     \end{array}
              \right]}
\end{eqnarray*}
It is straightforward to verify that, for $1 \le i \le 2$,
\[ a_{i,2} + \sum_{j \equiv 0,1\ (\mod 4)} a_{i,j} = a_{i,1} + \sum_{j
\equiv 2,3\ (\mod 4)} a_{i,j}\ (\mbox{where } 3 \le j \le m)\]
and \[a_{i,1} < a_{i,2} < \ldots < a_{i, m}.\] Hence, for $1 \le i \le 2$,
 the $m$-tuple
$$(a_{i,1}, -a_{i,2}, a_{i,3}, -a_{i,5}, a_{i,7}, \ldots, a_{i,m-3},
-a_{i,m-1}, -a_{i,m-2}, a_{i,m-4}, -a_{i,m-6}, \ldots, a_{i,6},
-a_{i,4}, a_{i,m})$$ is a difference $m$-tuple and corresponds to an
$m$-cycle $C_i$ with $\ell(C_i) = \{a_{i,1}, a_{i,2}, \ldots,
a_{i,m}\}$.  Hence, $X' = \{C_1, C_2\}$ is a minimum generating set for a
cyclic $m$-cycle system of
$ \langle S' \setminus B \rangle_{5m}$.
\end{proof}

We are now ready to prove the main result of this subsection, namely, that $K_{mt} - I$ has a  cyclic $m$-cycle system for every $t\equiv 1\ (\mod 4)$ and $m \equiv
2\ (\mod 8)$ with $t > 1$ if $m = 2p^{\alpha}$ for some prime $p$
and integer $\alpha \ge 1$.

\begin{lemma} \label{t=1(mod4)}
For each positive integer $t\equiv 1\ (\mod 4)$ and each $m \equiv
2\ (\mod 8)$ with $t > 1$ if $m = 2p^{\alpha}$ for some prime $p$
and integer $\alpha \ge 1$, there exists a cyclic $m$-cycle system
of $K_{mt} - I$.
\end{lemma}

\noindent
\begin{proof}
Let $m$ and $t$ be positive integers such that $m \equiv 2\ (\mod
8)$ and $t \equiv 1\ (\mod 4)$. Thus $m = 8r+2$ for some positive
integer $r$. Then $K_{mt} -I = \langle S' \rangle_{mt}$ where $S' = \{1, 2,
\ldots, (mt-2)/2\}$. Since $K_m - I$ has a cyclic hamiltonian cycle
system \cite {GM} if and only if $m \not= 2p^{\alpha}$ for some
prime $p$ and integer $\alpha \ge 1$, we may assume that $t > 1$.
Thus, let $t = 4s+1$ for some positive integer $s$.  By Lemmas \ref{m=10} and \ref{t=5}, we may assume that $s \ge 2$ and $r \ge 2$.

The proof
proceeds as follows. We begin by finding a set $B \subseteq S'$ such
that $|B|=4r$ and $\langle B \rangle_{mt}$ has a cyclic $m$-cycle system with
a minimum generating set $X$ consisting of cycles each with two
distinct lengths and orbit $2t$. We then construct an $(|S'
\setminus B| / m) \times m$ array $A = [a_{i,j}]$ with the property
that for each $i$ with $1\le i \le |S' \setminus B|/m$,
\[ a_{i,2} + \sum_{j \equiv 0,1\ (\mod 4)} a_{i,j} = a_{i,1} + \sum_{j
\equiv 2,3\
(\mod 4)} a_{i,j}\ (\mbox{where } 3 \le j \le m)\]
and \[a_{i,1} < a_{i,2} < \ldots < a_{i, m}.\] Thus for
each $i = 1, 2, \ldots, |S' \setminus B|/m$, the $m$-tuple
$$(a_{i,1}, -a_{i,2}, a_{i,3}, -a_{i,5}, a_{i,7}, \ldots, a_{i,m-3},
-a_{i,m-1}, -a_{i,m-2}, a_{i,m-4}, -a_{i,m-6}, \ldots, a_{i,6},
-a_{i,4}, a_{i,m})$$ is a difference $m$-tuple and corresponds to an
$m$-cycle $C_i$ with $\ell(C_i) = \{a_{i,1}, a_{i,2}, \ldots,
a_{i,m}\}$.  Hence, $X' = \{C_1, C_2, \ldots,
C_{|S' \setminus B|/m}\}$ will be a minimum generating set for a cyclic
$m$-cycle system of $\langle S' \setminus B \rangle_{mt}.$

Let $w=\lfloor r/2\rfloor$, and let $\delta_r=2(r/2-w)$, so that $\delta_r=1$ if $r$ is odd and $\delta_r=0$ if $r$ is even. Write $w
= qs+b$ where $q$ and $b$ are non-negative integers with $0 \le b <
s$ (note that it may be the case that $q=0$). For integers $i$ and
$j$, define $d_{i,j} = 4(r-2i)t + j$. Consider the path $P_{i,j} :
0, d_{i,j}, 4t$ and observe that $\ell(P_{i,j}) = \{4(r-2i)t+j,
4(r-2i-1)t+j\}$.  If $0< j <t$, then $C_{i,j} = P_{i,j} \cup
\rho^{2t}(P_{i,j}) \cup \rho^{4t}(P_{i,j}) \cup \cdots \cup
\rho^{(m-2)t}(P_{i,j})$  is an $m$-cycle since $m \equiv 2\ (\mod
8)$ gives $\gcd(4t, mt) = 2t$.  Thus, if
$0 < j < t$, $\ell(C_{i,j}) = \{4(r-2i)t+j, 4(r-2i-1)t+j\}$. Let
$$X = \{C_{i,j} \mid 0 \le i \le q-1 \mbox{ and } 1 \le j \le t-1\} \cup
\{C_{q,j} \mid  t-4b-2\delta_r \le j \le t-1\}$$ and let
\begin{eqnarray*}
B &=& \{4(r-2i)t+j, 4(r-2i-1)t+j \mid 0 \le i \le q-1 \mbox{ and } 1
\le j \le
t-1\} \\
&&\cup \{4(r-2q)t+j, 4(r-2q-1)t+j \mid  t-4b \le j \le t-1\},
\end{eqnarray*}
where we take the appropriate sets to be empty if $q=0$ or $b = 0$.
Observe that $X$ is a minimum generating set for $\langle B \rangle_{mt}$, and
consider the set $S' \setminus B$.  Now $|X| = 4qs + 4b$ so that
$|B| = 2(4qs + 4b) = 4r$.  Hence $|S' \setminus B| = 4(r+1)t-1 -4r =
2s(8r+2)$ and
\begin{eqnarray*}
S' \setminus B &=& \{1, 2, \ldots, 4(r-2q-1)t + t-1 -2\delta_r-4b\} \\
&& \cup \  \{4(r-2q-1)t + t, 4(r-2q-1)t + t + 1, \ldots, 4(r-2q)t +
t-1 -2\delta_r-4b\}\\
&& \cup \  \{ 4kt+t, 4kt+t+1, \ldots, 4(k+1)t \mid r-2q \le k \le
r-1\}.
\end{eqnarray*}
Note that $S'\setminus B$ has been written as the disjoint union of sets, each of which
has even cardinality and consists of consecutive integers.

The smallest length in $B$ is $4(r-2q-1)t + t-4b-2\delta_r$, and we wish to show this length is
at least $12s+1$.  Now $r \ge 2w = 2(qs+b) > 2q+1$ since $s \ge 2$.
Next since $0 \le b < s$ and $t = 4s+1$, we have
$t-1-4b = 4s - 4b \ge 4$. Therefore, $4(r-2q-1)t \ge 4t
> 16s$, and thus $4(r-2q-1)t + t-3-4b >16s+2> 12s$.
Since the smallest length is $S' \setminus B$ is at least $12s+1$
and since $S' \setminus B$ consists of sets of consecutive integers
of even cardinality, we may partition $S' \setminus B$ into sets $T,
S_1, \ldots, S_{8rs-4s}$ where $T = \{1, 2, \ldots, 12s\}$, and for
$i = 1, 2, \ldots, 8rs-4s$, $S_i = \{b_{i}, b_{i}+1,\}$ with $b_{1}
< b_{2} < \cdots < b_{8rs-4s}.$ Let $A = [a_{i,j}]$ be the $2s
\times m$ array
\begin{eqnarray*} {\footnotesize
\left[ \begin{array}{lllllllllllllll}
1 & 2 & 3 & 4 & 8s+1 & 8s+3 & b_1 & b_1+1 &  \\
5 & 6 & 7 & 8 & 8s+2 & 8s+4 & b_{4r-1} & b_{4r-1}+1 &  \\
\vdots & \vdots & \vdots & \vdots & \vdots & \vdots & \vdots & \vdots &
\cdots \\
8s-3 & 8s-2 & 8s-1 & 8s &12s-2 & 12s & b_{8rs-4s-4r+3} & b_{8rs-4s-4r+3}+1 &
\\
                     \end{array}
              \right.}
\end{eqnarray*}
\begin{eqnarray*} {\footnotesize
\left. \begin{array}{llllllllllllllllll}
b_2 & b_2+1 & \cdots & b_{4r-2} & b_{4r-2} + 1\\
b_{4r} & b_{4r}+1 &  \cdots & b_{8r-4} & b_{8r-4} + 1 \\
\vdots & \vdots & \cdots & \vdots &\vdots  \\
b_{8rs-4s-4r+4} & b_{8rs-4s-4r+4}+1 & \cdots & b_{8rs-4s} & b_{8rs-4s}+1 \\
                     \end{array}
              \right].}
\end{eqnarray*}
Clearly, for each $i$ with $1\le i \le 2s$,
\[ a_{i,2} + \sum_{j \equiv 0,1\ (\mod 4)} a_{i,j} = a_{i,1} + \sum_{j
\equiv 2,3\
(\mod 4)} a_{i,j}\ (\mbox{where } 3 \le j \le m)\]
and \[a_{i,1} < a_{i,2} < \ldots < a_{i, m}.\] Thus the $m$-tuple
$$(a_{i,1}, -a_{i,2}, a_{i,3}, -a_{i,5}, a_{i,7}, \ldots, a_{i,m-3},
-a_{i,m-1}, -a_{i,m-2}, a_{i,m-4}, -a_{i,m-6}, \ldots, a_{i,6},
-a_{i,4}, a_{i,m})$$ is a difference $m$-tuple and corresponds to an
$m$-cycle $C_i$ with $\ell(C_i) = \{a_{i,1}, a_{i,2}, \ldots,
a_{i,m}\}$.  Hence, $X' = \{C_1, C_2, \ldots,
C_{2s}\}$ is a minimum generating set for a cyclic $m$-cycle system of
$ \langle S' \setminus B \rangle_{mt}$.
\end{proof}

\subsection{The Case when $m \equiv 6 \ (\mod 8)$ and $t \equiv 3\ (\mod 4)$} \hspace{1in}

\vspace{12pt}

In this subsection, we find a cyclic $m$-cycle system of $K_{mt}-I$ when $m \equiv 6 \ (\mod 8)$ and $t \equiv 3\ (\mod 4)$.  We begin with three special cases, namely when $m = 6$, $m=14$, or $t = 3$. We first consider the case $m = 6$.

\begin{lemma} \label{m=6}
For all positive integers $t\equiv 3\ (\mod 4)$, there exists a cyclic $6$-cycle system of $K_{6t} - I$.
\end{lemma}

\begin{proof}
Let $t$ be a positive integer such that $t \equiv 3 \ (\mod 4)$,  say $t=4s+3$ for
some non-negative integer $s$.  Then $K_{6t} - I = \langle S' \rangle_{6t}$ where $S'=\{1,2, \ldots, 12s+8\}$.

Consider the
paths $P_i : 0, 3t - i, 2t, \mbox{ for } 1 \leq i \leq 4$; then
$\ell (P_i) = \{ 3t - i, t - i \}$.  Next, let $C_i = P_i \cup
\rho^{2t} (P_i) \cup \rho^{4t} (P_i)$.  Then each $C_i$ is a 6-cycle
and $X = \{C_1, C_2,C_3,C_4 \}$ is a minimum generating set for
$\langle B \rangle_{6t}$ where $B=\{ 3t-i, t-i \mid 1 \leq i \leq 4 \}$. Now,
$t-5 = 4s-2$ and thus $S' \setminus B = \{ 1, 2, \ldots, 4s-2, 4s+3,
4s+4, \ldots, 12s+4 \}$, and so we must find a cyclic 6-cycle system
of $\langle S' \setminus B \rangle_{6t}$. Let $A = [a_{i,j}]$ be the $2s
\times 6$ array
\begin{eqnarray*} {\footnotesize
\left[ \begin{array}{lllllllllllllll}
1 & 2 & 3 & 4 & 8s+5 & 8s+7 \\
5 & 6 & 7 & 8 & 8s+6 & 8s+8 \\
\vdots & \vdots & \vdots & \vdots & \vdots & \vdots\\
4s-3 & 4s-2 & 4s+3 & 4s+4 & \alpha & \alpha + 2 \\
\vdots & \vdots & \vdots & \vdots & \vdots & \vdots\\
8s+1 & 8s+2 & 8s+3 & 8s+4 & 12s+2 & 12s+4\\
                     \end{array}
              \right]}
\end{eqnarray*}
where
\begin{equation*}
\alpha=\begin{cases} 10s+2 & \text{if s is even,}\\
                    10s+3 & \text{if s is odd.}
                    \end{cases}
\end{equation*}
Clearly, for each $i$ with $1\le i \le 2s$,
\[ a_{i,2} + \sum_{j \equiv 0,1\ (\mod 4)} a_{i,j} = a_{i,1} + \sum_{j
\equiv 2,3\ (\mod 4)} a_{i,j}\quad (\mbox{where } 3 \le j \le 6)\] and
\[a_{i,1} < a_{i,2} < \ldots < a_{i, 6}.\] Thus the $6$-tuple
$$(a_{i,1}, -a_{i,2}, a_{i,3}, -a_{i,4}, -a_{i,5},
a_{i,6})$$ is a difference $6$-tuple and corresponds to a $6$-cycle
$C'_i$ with $\ell(C'_i) = \{a_{i,1}, a_{i,2}, \ldots, a_{i,6}\}$.
Hence, $X' = \{C'_1, C'_2, \ldots, C'_{2s}\}$ is a minimum generating
set for a cyclic $6$-cycle system of $\langle S' \setminus B\rangle_{6t}$.
\end{proof}

Next we consider the case when $m=14$.

\begin{lemma} \label{m=14}
For all positive integers $t\equiv 3\ (\mod 4)$, there exists a cyclic $14$-cycle system of $K_{14t} - I$.
\end{lemma}

\begin{proof}
Let $t$ be a positive integer such that $t \equiv 3 \ (\mod 4)$,  say $t=4s+3$ for
some non-negative integer $s$.  Then $K_{14t} - I = \langle S' \rangle_{14t}$ where $S'=\{1,2, \ldots, 28s+20\}$.

Consider the
paths $P_i : 0, 7t - i, 2t, \mbox{ for } 1 \leq i \leq 10$; then
$\ell (P_i) = \{ 7t - i, 5t - i \}$.  Next, let $C_i = P_i \cup
\rho^{2t} (P_i) \cup \rho^{4t} (P_i) \cup \cdots \cup \rho^{12t}(P_i)$.  Then each $C_i$ is a $14$-cycle
and $X = \{C_1, C_2,\ldots ,C_{10} \}$ is a minimum generating set for
$\langle B \rangle_{14t}$ where $B=\{ 7t-i, 5t-i \mid 1 \leq i \leq 10 \}$. Now,
$5t-10 = 20s+5$ and thus $S' \setminus B = \{ 1, 2, \ldots, 20s+4, 20s+15,
20s+16, \ldots, 28s+10 \}$, and so we must find a cyclic 14-cycle system
of $\langle S' \setminus B \rangle_{14t}$. Let $A = [a_{i,j}]$ be the $2s
\times 14$ array
\begin{eqnarray*} {\footnotesize
\left[ \begin{array}{lllllllllllllll}
1 & 2 & 3 & 4 & 8s+1 & 8s+3 & 12s+1 & 12s+2  & 12s+3 & 12s+4& \\
5 & 6 & 7 & 8 & 8s+2 & 8s+4 & 12s+5 & 12s+6 & 12s+7 & 12s+8  &\\
9 & 10 & 11 & 12 & 8s+5 & 8s+7 & 12s+9 & 12s+10 & 12s+11 & 12s+12&  \\
\vdots & \vdots & \vdots & \vdots & \vdots & \vdots & \vdots & \vdots & \vdots & \vdots & \\
8s-3   & 8s-2   & 8s-1   & 8s     & 12s-2  & 12s    & 20s-3  & 20s-2  & 20s-1  & 20s &   \\
                     \end{array}
              \right.}
\end{eqnarray*}
\begin{eqnarray*} {\footnotesize
\left. \begin{array}{lllllllllllllll}
20s+1 & 20s+2 & 20s+3 & 20s+4\\
20s+15 & 20s+16 & 20s+17 & 20s+18 \\
20s+19 & 20s+20 & 20s+21 & 20s+22 \\
\vdots & \vdots & \vdots & \vdots\\
28s+7  & 28s+8  & 28s+9 & 28s+10\\
                     \end{array}
              \right].}
\end{eqnarray*}

Clearly, for each $i$ with $1\le i \le 2s$,
\[ a_{i,2} + \sum_{j \equiv 0,1\ (\mod 4)} a_{i,j} = a_{i,1} + \sum_{j
\equiv 2,3\ (\mod 4)} a_{i,j}\quad (\mbox{where } 3 \le j \le 14)\] and
\[a_{i,1} < a_{i,2} < \ldots < a_{i, 14}.\] Thus the $14$-tuple
$$(a_{i,1}, -a_{i,2}, a_{i,3}, -a_{i,5},
a_{i,7}, -a_{i,9}, a_{i,11}, -a_{i,13}, -a_{i,12}, a_{i,10}, -a_{i,8}, a_{i,6}, -a_{i,4}, a_{i,14})$$ is a difference $14$-tuple and corresponds to a $14$-cycle
$C'_i$ with $\ell(C'_i) = \{a_{i,1}, a_{i,2}, \ldots, a_{i,14}\}$.
Hence, $X' = \{C'_1, C'_2, \ldots, C'_{2s}\}$ is a minimum generating
set for a cyclic $14$-cycle system of $\langle S' \setminus B\rangle_{14t}$.
\end{proof}

We now consider the case when $t = 3$.

\begin{lemma} \label{t=3}
For all positive integers $m \equiv 6\
(\mod 8)$, there exists a cyclic $m$-cycle system of $K_{3m} - I$.
\end{lemma}

\begin{proof}
Let $m$ be a positive integer such that $m  \equiv 6 \ (\mod
8)$, say $m=8r + 6$ for
some non-negative integer $r$. By Lemmas \ref{m=6} and \ref{m=14}, we may assume $r \ge 2$. Then $K_{3m} - I = \langle S' \rangle_{mt}$ where $S'=\{1,2, \ldots, 12r+8\}$. Write $2r=4q+b+2 \mbox{ for integers } q \geq 0 \mbox{
and } b \in
\{0,2\}$, and let $a$ be a positive integer such that $1+ \log_2 (q+1) \leq a \leq 1+
\log_2 (3q+3/2+5b/6)$. For integers $i$ and
$j$, define $d_{i,j}=6(2r-i)+j$.  Then consider the path
$P_{i,j} : 0, d_{i,j},3\cdot 2^a$; so  $\ell(P_{i,j})=\{6(2r-i)+j,6(2r-i)+j-3\cdot
2^a\}$. Now, let
$C_{i,j}=P_{i,j} \cup \rho^{6}(P_{i,j})\cup \cdots \cup
\rho^{3(m-2)}(P_{i,j})$. Then $ C_{i,j}$ is an $m$-cycle since $m
\equiv 6 \ (\mod 8)$ implies $\gcd(3 \cdot 2^a,3m)=6$. Thus, $\ell
(C_{i,j})=\ell (P_{i,j})$.

Now, let
\begin{eqnarray*}
X & = & \{ C_{0,j} \mid  j=7,8 \}\\
&& \cup \  \{ C_{i,j} \mid 0 \leq i \leq q-1 \mbox{ and } 1 \leq j
\leq 4 \}\\
&& \cup \  \{ C_{q,j} \mid 5-b \leq j \leq 4\}\\
\end{eqnarray*}
and let
\begin{eqnarray*}
B & = & \{12r+7,12r+7-3\cdot 2^a,12r+8,12r+8-3\cdot2^a \}\\
&& \cup \  \{6(2r-i)+j,6(2r-i)-3\cdot 2^a+j \mid 0 \leq i \leq
q-1 \mbox{ and }1\leq j \leq 4\} \\
&& \cup \ \{6(2r-q)+j,6(2r-q)-3\cdot2^a+j \mid 5-b \leq j \leq 4\}\\
\end{eqnarray*}
where, if $q=0$ or $b=0$, we take the corresponding sets to be empty as necessary.
Now $B$ will consists of $4r$ distinct lengths and $X$ will be a minimum generating set for $\langle B \rangle_{3m}$ if $12r + 8 - 3 \cdot 2^a \le 6(2r-q) + 5-b -1$.  Note that $1+ \log_2 (q+1) \leq a$ so that $q+1 \le 2^{a-1}$. Next,
$$12r + 8 - [6(2r-q) + 5-b-1] = 6q + 4 + b \le 6q + 6 = 6(q+1) \le 6 \cdot 2^{a-1} = 3 \cdot 2^a,$$
and hence $12r + 8 - 3 \cdot 2^a \le 6(2r-q) + 5-b -1$.  Thus, $B$ consists of $4r$ distinct lengths, and $X$ is a minimum generating set for $\langle B \rangle_{3m}$.
Now, the smallest length in $B$ is $6(2r-q)+5-b-3 \cdot 2^a$ and we want this length
to be greater than $8$.  Recall that $a \leq 1+
\log_2 (3q+3/2+5b/6)$ and thus $2^{a-1} \le 3q + 3/2 + 5b/6$.  Hence, $3 \cdot 2^a \le 18q + 9 + 5b = 12r - 6q - 3 - b$ since $2r = 4q + b + 2$.  Therefore, $6(2r-q) + 5-b - 3 \cdot 2^a\ge 8$. Since $|B| = 4r$, we have $|S'\setminus B| = 8r + 8$.  Note that
\begin{eqnarray*}
S' \setminus B &=& \{1, 2, \ldots, 6(2r-q) + 5-b - 3 \cdot 2^a-1\}\\
&&  \cup \  \{6(2r-i) -3\cdot 2^a +5,  6(2r-i) -3\cdot 2^a +6 \mid 0 \le i \le q\}\\
&& \cup \ \{12r - 3 \cdot 2^a + 9, \ldots, 6(2r-q) + 5-b-1\} \\
&& \cup \ \{6(2r-i) +5, 6(2r-i)+6 \mid 0 \le i \le q\}.
\end{eqnarray*}
Note that $S'\setminus B$ has been written as the disjoint union of sets, each of which has even cardinality and consists of consecutive integers.
Therefore, we may partition $S' \setminus B$ into sets $T, S_1, S_2,
\ldots, S_{4r}$ where $T = \{1, 2, \ldots, 8\}$ and for $i = 1,
2, \ldots, 4r$, let $S_i = \{b_{i}, b_{i}+1\}$ with $b_{1} < b_{2}
< \cdots < b_{4r}.$

Consider the $m$-tuple
\begin{eqnarray*}
(1, -3, 6, -7, b_1, -b_2, b_3, -b_4, \ldots, b_{4r-1}, -b_{4r}, -(b_{4r-1}+1), b_{4r-2}+1,\\ -(b_{4r-3}+1), b_{4r-4}+1,
\ldots, b_2+1, -(b_1+1), 8, -5, b_{4r}+1)
\end{eqnarray*}
which is a difference $m$-tuple and corresponds to an
$m$-cycle $C_1$ with $$\ell(C_1) = \{1, 3, 5, 6, 7 ,8, b_1, b_1+1, b_2, b_2+1, \ldots, b_{4r}, b_{4r}+1\}.$$
 Then consider the path
$P : 0, 2,6$; so  $\ell(P)=\{2, 4\}$. Now, let
$C_2=P\cup \rho^{6}(P)\cup \cdots \cup
\rho^{3(m-2)}(P)$. Then $ C_2$ is an $m$-cycle since $m
\equiv 6 \ (\mod 8)$ implies $\gcd(6,3m)=6$. Thus, $\ell
(C_2)=\ell (P) = \{2, 4\}$.
 Hence, $X' = \{C_1, C_2\}$ is a minimum generating set for a
cyclic $m$-cycle system of
$ \langle S' \setminus B \rangle_{3m}$.
\end{proof}

We now prove the main result of this subsection, namely that $K_{mt} - I$ has a cyclic $m$-cycle system for every $t\equiv 3\ (\mod 4)$ and  $m \equiv 6\
(\mod 8)$.

\begin{lemma}
For all positive integers $t\equiv 3\ (\mod 4)$ and  $m \equiv 6\
(\mod 8)$, there exists a cyclic $m$-cycle system of $K_{mt} - I$.
\end{lemma}

\begin{proof}
Let $m$ and $t$ be positive integers such that $m  \equiv 6 \ (\mod
8)$ and $t \equiv 3 \ (\mod 4).$ Then $m=8r + 6$ and $t=4s+3$ for
some non-negative integers $ r$ and $s$.  Then $K_{mt} - I = \langle S' \rangle_{mt}$ where $S'=\{1,2, \ldots, (4r+3)t -1\}$.

By Lemmas \ref{m=6}, \ref{m=14}, and \ref{t=3}, we may assume $s\geq 1$ and $r\geq 2$. First, write
$6r+4=(2t-2)q+(t-1)\ell+b$ for integers $q , \ell \mbox{ and } b$
with $q \geq 0, 0 \leq b < 2t-2$ , and $\ell=0$ if $6r+4 < t-1$, or
$\ell=1$ otherwise. For integers $i$ and $j$, define $d_{i,j} =
2t(2r-2i-1) + j$. Consider the path $P_{i,j} : 0, d_{i,j}, 2t$ and
note that $\ell(P_{i,j}) = \{2t(2r-2i-1)+j, 2t(2r-2i-2)+j \}$. If
$0< j < 2t$, then $C_{i,j} = P_{i,j} \cup \rho^{2t}(P_{i,j}) \cup
\rho^{4t}(P_{i,j}) \cup \cdots \cup \rho^{(m-2)t}(P_{i,j})$  is an
$m$-cycle since $m \equiv 6\ (\mod 8)$
implies $\gcd(2t, mt) = 2t$. Thus, if $0 < j < 2t$, then
$\ell(C_{i,j}) = \ell(P_{i,j})$.

Now, let
\begin{eqnarray*}
X & = & \{ C_{-1,j} \mid 1 \leq j \leq t-1 \}\\
&& \cup \  \{ C_{i,j} \mid 0 \leq i \leq q-1 \mbox{ and } 1 \leq j
\leq 2t-2 \}\\
&& \cup \  \{ C_{q,j} \mid 2t-1-b \leq j \leq 2t-2\}\\
\end{eqnarray*}
and let
\begin{eqnarray*}
B & = & \{ 2t(2r+1) + j, 2t(2r) + j \mid 1 \leq j \leq t-1 \}\\
&& \cup \  \{2t(2r-2i -1) +j, 2t(2r-2i-2) +j \mid 1 \leq j \leq
2t-2 \mbox{ and } 0\leq i \leq q-1\} \\
&& \cup \ \{2t(2r-(2q+1))+2t-2-j, 2t(2r-2q-2) +2t -2 -j \mid 0 \leq
j \leq b-1\}.
\end{eqnarray*}
where we take the first set to be empty if $\ell=0$, the second to be empty if $q=0$, and the third to be empty if $b=0$. Then $X$
is a minimum generating set for $ \langle B \rangle_{mt}$.

Now we must find a cyclic $m$-cycle system of $ \langle S'\setminus
B \rangle_{mt}$. First, $|B|=2[(2t-2)q +(t-1)\ell +b] = 12r+8$ so that $|S'
\setminus B| =(4r+3)t - 1 -12r-8=(8r+6)(2s)$. Moreover,
\begin{eqnarray*}
S' \setminus B & = & \{1,2, \ldots, 2t(2r-2q-1)-b-2\}\\
&& \cup \ \{2t(2r-2q-1)-1, 2t(2r-2q-1), \ldots,2t(2r-2q)-b-2\}\\
&& \cup \ \{2t(2r-i)-1,2t(2r-i) \mid 0\leq i \leq 2q\} \\
&& \cup \ \{4rt+t,4rt+t+1,\ldots ,4rt+2t\}.\\
\end{eqnarray*}
The smallest length in $B$ is $4t(r-q-1) + (2t-1) -b$, and we must
verify that this length is at least $12s+1$. Note that we have $2t-1-b > 1$. Thus, it is sufficient
to prove that $4t(r-q-1) \geq 12s$, or $t(r-q-1) \geq 3s$. This
inequality follows if $r>q+1$. Clearly, this is true if $q=0$ since $r \ge 2$, so
assume $q \geq 1$. Then $\ell=1$, and so $6r+4=2q(4s+2)+(4s+2)+b$,
or
\begin{eqnarray*}
3r+2 & = &q(4s+2)+2s+1+b/2\\
&=&4qs+2q+2s+1+b/2\\
& \geq & 6q+3  \hspace{.1in}(\mbox{since } s \geq 1).\\
\end{eqnarray*}
So, $r \geq 2q+ 1/3 > q+1$ since $q \geq 1$. Since the smallest
length in $B$ is at least $12s+1$ and $S' \setminus B$ consists of
sets of consecutive integers of even cardinality, we may partition
$S' \setminus B$ into sets $T,S_1,\ldots,S_{8rs}$ where
$T=\{1,2,\ldots,12s\}$, and for $i=1,2,\ldots, 8rs,
S_i=\{b_i,b_i+1\}$ with $b_1\leq b_2\leq \cdots \leq b_{8rs}$. Let
$A = [a_{i,j}]$ be the $2s \times m$ array
\begin{eqnarray*} {\footnotesize
\left[ \begin{array}{lllllllllllllll}
1 & 2 & 3 & 4 & 8s+1 & 8s+3 & b_1 & b_1+1 &  \\
5 & 6 & 7 & 8 & 8s+2 & 8s+4 & b_{4r+1} & b_{4r+1}+1 &  \\
\vdots & \vdots & \vdots & \vdots & \vdots & \vdots & \vdots &
\vdots &
\cdots \\
8s-3 & 8s-2 & 8s-1 & 8s &12s-2 & 12s & b_{8rs-4r+1} & b_{8rs-4r+1}+1
&
\\
                     \end{array}
              \right.}
\end{eqnarray*}
\begin{eqnarray*} {\footnotesize
\left. \begin{array}{llllllllllllllllll}
b_2 & b_2+1 & \cdots & b_{4r} & b_{4r} + 1\\
b_{4r+2} & b_{4r+2}+1 &  \cdots & b_{8r} & b_{8r} + 1 \\
\vdots & \vdots & \cdots & \vdots &\vdots  \\
b_{8rs-4r+2} & b_{8rs-4r+2}+1 & \cdots & b_{8rs} & b_{8rs}+1 \\
                     \end{array}
              \right].}
\end{eqnarray*}

Clearly, for each $i$ with $1\le i \le 2s$,
\[ a_{i,2} + \sum_{j \equiv 0,1\ (\mod 4)} a_{i,j} = a_{i,1} + \sum_{j
\equiv 2,3\ (\mod 4)} a_{i,j}\ (\mbox{where } 3 \le j \le m)\] and
\[a_{i,1} < a_{i,2} < \ldots < a_{i, m}.\] Thus the $m$-tuple
$$(a_{i,1}, -a_{i,2}, a_{i,3}, -a_{i,5}, a_{i,7}, \ldots, a_{i,m-3},
-a_{i,m-1}, -a_{i,m-2}, a_{i,m-4}, -a_{i,m-6}, \ldots, a_{i,6},
-a_{i,4}, a_{i,m})$$ is a difference $m$-tuple and corresponds to an
$m$-cycle $C_i$ with $\ell(C_i) = \{a_{i,1}, a_{i,2}, \ldots,
a_{i,m}\}$.  Hence, $X' = \{C_1, C_2, \ldots, C_{2s}\}$ is a minimum
generating set for a cyclic $m$-cycle system of $ \langle S'
\setminus B \rangle_{mt}$.
\end{proof}

\end{document}